\documentclass[10pt,draft,reqno]{amsart}
\usepackage[cp1251]{inputenc}
\usepackage[T2A]{fontenc}
\usepackage[english]{babel}

\usepackage{amsmath,amssymb,amsfonts,amsthm}

     \makeatletter
     \def\section{\@startsection{section}{1}%
     \z@{.7\linespacing\@plus\linespacing}{.5\linespacing}%
     {\bfseries
     \centering
     }}
     \def\@secnumfont{\bfseries}
     \makeatother

\usepackage{graphicx,color}
\usepackage{xcolor}
\colorlet{blu1}{blue!70!black}
\colorlet{red1}{red!80}

\newtheorem{theorem}{Theorem}[section]
\newtheorem{lemma}[theorem]{Lemma}

\newtheorem{corollary}[theorem]{Corollary}
\theoremstyle{definition}
\newtheorem{definition}[theorem]{Definition}
\newtheorem{example}[theorem]{Example}
\theoremstyle{remark}
\newtheorem{remark}[theorem]{Remark}
\numberwithin{equation}{section}
\begin{document}

\title[Measures and Trajectory Properties in Oscillator Systems]{Measures and Trajectory Properties in Oscillator Systems
}

\author{Vsevolod Sakbaev*}
\address{Vsevolod Sakbaev: Steklov Mathematical Institute of Russian Academy of Science, Gubkin str. 8, Moscow 119991, Russia}
\email{fumi2003@mail.ru}

\author[Igor Volovich]{Igor Volovich*}
\thanks{* This research was founded by Russian Scientific Foundation under grant 19-11-00320.}
\address{Igor Volovich: 
Steklov Mathematical Institute of Russian Academy of Science, Gubkin str. 8, Moscow 119991, Russia}
\email{volovich@mi-ras.ru}

\subjclass[2010] {Primary 28D05, 37A05, 47A35; Secondary 37K10, 37N20}

\keywords{Linear flow on infinite-dimensional tori; periodic trajectory; transitive trajectory; non-wandering points; characteristic function of a measure.}

\begin{abstract} 
This paper investigates the properties of trajectories in harmonic oscillator systems equipped with a point, absolutely continuous, or singular measure.  

As demonstrated in \cite{VS24}, infinite-dimensional linear flows of countable oscillator systems exhibit a new class of trajectory behavior. Specifically, these trajectories are non-periodic, and their projections onto any four-dimensional symplectic subspace fail to be dense in the corresponding projection of the invariant torus. Such trajectories do not arise in finite-dimensional systems, are non-generic for countable oscillator systems, but become generic in the continual case.

We prove that for a countable harmonic oscillator system, every point on a non-degenerate invariant torus is a non-wandering point of the flow. In contrast, for a continual system with an absolutely continuous measure, all points on such a torus are wandering. Furthermore, for continual systems with a singular measure, we establish sufficient conditions on the measure and torus that rule out the existence of both transitive trajectories and non-wandering points. As an application, we exhibit a class of singular Bernoulli measures satisfying these conditions.
\end{abstract}

\maketitle

\section{Introduction}\label{introd}
Infinite systems of harmonic oscillators constitute an important class of infinite-dimensional Hamiltonian systems.  

In particular, every isolated quantum system can be represented by a unitary group in a complex Hilbert space. Since any such unitary group is mathematically equivalent to a system of harmonic oscillators, it follows that all isolated quantum systems are unitarily equivalent to classical harmonic oscillator systems \cite{Volovich21}.

A rigorous description of infinite-dimensional Hamiltonian systems relevant to this work is provided in the Appendix (see also \cite{Chernoff}). 

Notably, any Hamiltonian system with an invariant measure on its phase space can be interpreted as a system of harmonic oscillators. This follows from the Koopman-von Neumann theory, which states that when a Hamiltonian system's phase space is equipped with a flow-invariant measure, the corresponding Hamiltonian flow induces a group of unitary operators on the space of square-integrable functions with respect to this measure \cite{Khrennikov, KS-18, Volovich19, S23}.

A unitary group acting on a Hilbert space can be represented as the Hamiltonian flow of a harmonic oscillator system in a real representation of the Hilbert space endowed with a shift-invariant symplectic form \cite{Volovich21}. Here, a real representation of a complex Hilbert space refers to equipping the space with a real structure (or equivalently, a K$\ddot{\rm a}$hler structure) \cite{Khrennikov, KS-18, V25}.

The Hamiltonian flow of an oscillator system constitutes a Liouville-integrable system. Such systems admit a complete family of first integrals in involution, whose level surfaces form invariant tori \cite{Volovich21}.
An infinite dimentional ascillator systems arise in considerations of quantum random walks \cite{BPS24}.

An oscillator system is defined as a parameterized collection of non-interacting two-dimensional oscillators equipped with a measure on the parameter space. A system trajectory is represented by a square-integrable function on the parameter space, where the function's value at each parameter corresponds to the trajectory of an individual oscillator in its two-dimensional phase subspace.

The properties of these trajectories depend fundamentally on the type of measure associated with the oscillator system.

The Hamiltonian flow (linear flow) of oscillator systems on invariant tori in finite-dimensional symplectic spaces has been extensively studied in the literature \cite{Che, Dumas, 21, KH}. Related work includes investigations of complete systems of conserved functionals for the periodic sine-Gordon equation and their associated infinite-dimensional tori \cite{Schw}, as well as developments in infinite-dimensional KAM theory \cite{Che, KAM, Kuksin}.

In this work, we examine fundamental dynamical properties of such linear flows, including:
\begin{itemize}
    \item Existence or non-existence of periodic trajectories
    \item Presence or absence of topologically transitive trajectories  
    \item Occurrence or non-occurrence of non-wandering points
\end{itemize}

The ergodicity of linear flows on finite-dimensional tori is characterized by Weyl's theorem \cite{Che}. This result was extended to infinite-dimensional Hilbert spaces by von Neumann \cite{IvN}, while Segal \cite{Segal} studied ergodicity for subgroups of the orthogonal group in this setting. Recent work in \cite{KVV21} has further generalized Weyl's ergodicity theorem to infinite-dimensional cases.

Of particular relevance to our work is the ergodic theorem for infinite-dimensional linear flows on invariant tori equipped with Kolmogorov measure, established in \cite{VS24}. 

In this paper, we investigate these ergodic properties for both countable and continual systems of harmonic oscillators.

The system of $N$ harmonic oscillators constitutes a Hamiltonian system in the $2N$-dimensional Euclidean space $E = \mathbb{R}^{2N}$ endowed with a shift-invariant symplectic form. In the corresponding symplectic coordinates $(q,p) \in E$, the Hamiltonian function takes the form
\begin{equation}\label{HN}
H(q,p) = \sum_{j=1}^N \frac{\lambda_j}{2}(q_j^2 + p_j^2),
\end{equation}
where $\{\lambda_1, \ldots, \lambda_N\}$ are real constants.

The harmonic oscillator system \eqref{HN} possesses $N$ first integrals given by $I_j(q,p) = q_j^2 + p_j^2$ for $j = 1, \ldots, N$. For any collection of $N$ positive numbers $r = \{r_1, \ldots, r_N\}$, the level surface 
\begin{equation}
{\mathbb T}_r = \{(q,p) \in E : I_j(q,p) = r_j \text{ for } j = 1, \ldots, N\}
\end{equation}
forms a smooth $N$-dimensional manifold that remains invariant under the Hamiltonian flow. This manifold is referred to as an invariant torus of the harmonic oscillator system \eqref{HN}. The resulting Hamiltonian flow restricted to this invariant torus is called a linear flow \cite{KH}.


\begin{definition}\label{def:commensurate}
A set $A \subset \mathbb{R}$ is called \emph{rationally commensurate} if there exists a finite subset $\{a_1,\ldots,a_n\} \subset A$ such that every element of $A$ can be expressed as a finite linear combination of $\{a_1,\ldots,a_n\}$ with rational coefficients \cite{AA, KSF, KH}.
\end{definition}

\begin{definition}\label{def:strongly-commensurate} 
A set $A \subset \mathbb{R}$ is called \emph{strongly rationally commensurate} if for every finite subset $\{a_1,\ldots,a_m\} \subset A$ with $m \geq 2$, each $a_i$ ($1 \leq i \leq m$) can be written as a linear combination of the remaining $m-1$ elements with rational coefficients.
\end{definition}

\begin{remark}\label{rmk:commensurability}
For a finite set $A \subset \mathbb{R}$, the following equivalences hold:
\begin{enumerate}
    \item The rational commensurability of $A$ is equivalent to the resonance condition for $A$ (see \cite[page 4]{21}).
    
    \item Using the notion of resonance with finite multiplicity introduced in \cite{21}, we can characterize strong commensurability: A finite set $A$ of $N$ numbers is strongly rationally commensurate if and only if it forms a resonant set of maximum multiplicity $N-1$.
\end{enumerate}
\end{remark}

\begin{remark}
The condition for rational commensurability of a finite set of numbers $A$ is equivalent to the condition for the resonance of a set of numbers $A$ (see \cite{21}, page 4).
Defining the resonance of finite multiplicity of a set of numbers $A$ in \cite{21} makes it possible to present conditions for strong commensurability of a finite set of numbers $A$.
A set of $N$ numbers $A$ is strongly commensurate if and only if it is a resonant set of maximum multiplicity $N-1$.
\end{remark}

The following result is known as Jacobi's theorem (see \cite[page 114]{AA}):

\begin{theorem}[Jacobi]\label{Yak}
Consider the Hamiltonian system (\ref{HN}) with frequencies $\{\lambda_j\}_{j=1}^N$.

\begin{enumerate}
    \item A trajectory is periodic if and only if the frequency set $\{\lambda_j\}_{j=1}^N$ is strongly rationally commensurate.
    
    \item A trajectory is dense on the $N$-dimensional invariant torus $T = \{I_k = c_k\}_{k=1}^N$ if and only if the frequencies satisfy the non-resonance condition:
    \begin{equation}\label{10}
    \sum_{k=1}^N \lambda_k n_k = 0 \quad \text{with} \quad n_k \in \mathbb{Z} \implies n_k = 0 \quad \forall k=1,\ldots,N.
    \end{equation}
\end{enumerate}
\end{theorem}

Theorem 1.2 is also known as the Weyl-Kronecker theorem \cite{KVV21}.

In \cite{VS24}, infinite-dimensional analogues of Jacobi's theorem \cite{AA} are established. These results provide criteria for both periodicity and topological transitivity of trajectories on invariant tori in infinite systems of harmonic oscillators.

According to Weyl's theorem \cite{Che}, a linear flow on a $d$-dimensional torus is ergodic with respect to the Lebesgue measure on the invariant torus if and only if its frequency vector satisfies the rational independence condition. This result has been extended to infinite-dimensional invariant tori in \cite{KVV21, VS24}.

A straightforward generalization of Jacobi's theorem to countable oscillator systems would suggest that: 
\begin{quote}
A trajectory of a countable oscillator system is periodic if and only if its frequency collection is strongly rationally commensurate.
\end{quote}
However, this statement fails in infinite-dimensional settings, as demonstrated in \cite{VS24} and Example 1 below. Specifically, there exist countable oscillator systems with rationally commensurate frequency collections that nevertheless admit no periodic trajectories. This phenomenon is inherently infinite-dimensional and cannot occur in finite-dimensional phase spaces.

Finite oscillator systems exhibit a fundamental dichotomy: 
\begin{enumerate}
    \item Every trajectory is periodic, \emph{or}
    \item The system contains a subsystem of at least two oscillators whose trajectory is dense in the subsystem's invariant torus \cite{AA, Che}.
\end{enumerate}

In infinite-dimensional phase spaces, linear flows exhibit at least three distinct types of trajectories. Theorem 5.5 (illustrated in Example 5.4) demonstrates the existence of linear flows that possess neither periodic trajectories nor trajectories whose projections onto any four-dimensional symplectic subspace are dense in the corresponding projection of the invariant torus. 

This establishes the following classification of linear flow trajectories on infinite-dimensional tori:
\begin{enumerate}
    \item \textbf{Type I (Periodic)}: Trajectories that are periodic;
    
    \item \textbf{Type II (Projectively dense)}: Non-periodic trajectories whose projections are dense in the projection of invariant torus to some four-dimensional symplectic subspace;
    
    \item \textbf{Type III (Novel type)}: Non-periodic trajectories whose projections are nowhere dense in any four-dimensional symplectic subspace.
\end{enumerate}

The existence of Type III trajectories constitutes a fundamentally new phenomenon in infinite-dimensional systems, with no finite-dimensional analogue.

Results of of the present paper were announced in \cite{VS-Arch}. 
The paper is organized as follows:

\begin{itemize}
    \item[\S2] We present fundamental definitions and preliminary results on harmonic oscillator systems, establishing the necessary theoretical framework.

    \item[\S3] We derive a necessary condition for the existence of transitive trajectories, expressed in terms of the characteristic function of the oscillator system's measure (see (\ref{eq1})).

    \item[\S4] For systems equipped with a Bernoulli measure (notably including singular measures with respect to Lebesgue measure), we prove the absence of dense trajectories in the invariant torus.

    \item[\S5] We establish criteria for periodicity and transitivity of trajectories in countable oscillator systems.

    \item[\S6] We investigate the periodicity and transitivity properties for continual oscillator systems.

    \item[\S7] We demonstrate that:
    \begin{itemize}
        \item Trajectories of countable systems are non-wandering
        \item Trajectories of continual systems are wandering
    \end{itemize}

    \item[\S8] We summarize and discuss the paper's main results.

    \item[\S9] (Appendix) We provide supplementary material on infinite-dimensional Hamiltonian systems.
\end{itemize}

\section {Representation of Hamiltonian dynamics in the form of dynamics of a system of oscillators}

Following \cite{Volovich19}, any Hamiltonian system admitting an invariant measure possesses a completely Liouville-integrable representation. More precisely, there exists an isomorphism between the original Hamiltonian system in its Koopman representation and
a Hamiltonian system consisting of a family of harmonic oscillators.

{ 
A system of harmonic oscillators (SHO) is a triplet (see \cite{Volovich19})
\begin{equation}\label{eq1}
(X,\mu, {\bf V}):
\end{equation}
where $X$ is a locally compact space; $\mu$ is a Borel countable additive {\it regular} {\it nonnegative} (Radon) measure on the space $X$;
$\bf V$ is {a $1$--parameter, strongly continuous sub--group of the}
special unitary group in the Hilbert space $L_2(X,\mu ,{\mathbb C})\equiv \mathcal H$,
${\bf V}_t: L_2(X,\mu )\to L_2(X,\mu ),\quad t\in {\mathbb R},$
{is defined} by the equality
\begin{equation}\label{2}
{\bf V}_t\phi (x)=e^{it\lambda (x)}\phi (x),\ t\in {\mathbb R} ,\, x\in X.
\end{equation}
Here $\lambda:\ X\to R$ is a measurable real-valued function, {and}
$\phi \in L_2(X,\mu )$.

When $\mu$ is a continuous measure, the system represents a field of harmonic oscillators.
When $\mu$ is a discrete (point) measure, the system reduces to a countable collection of harmonic oscillators.

In this work, we focus on the special case of SHO where $X = \mathbb{R}$. This choice corresponds to systems with simple spectrum of the generator of the unitary group $\mathbf{V}$. For such systems, there exists a canonical triplet $(X, \mu, \mathbf{V})$ with $X = \mathbb{R}$ and unitary evolution given by
\begin{equation*}
\mathbf{V}(t)u(x) = e^{itx}u(x), \quad t, x \in \mathbb{R}
\end{equation*}
\cite[Chapter 7.2]{RS}.

We remark that for general SHO systems, the phase space admits an orthogonal decomposition into invariant subspaces under the SHO flow. Furthermore, the restriction of the flow to each invariant subspace yields an SHO system with simple spectrum.

According to the Theorem 1.14 {\cite{RS}},
any Borel countably additive measure $\mu$ on $\mathbb{R}$ admits a unique decomposition into three mutually singular components:
\begin{equation*}
\mu = \mu_p + \mu_{ac} + \mu_s,
\end{equation*}
where:
\begin{itemize}
    \item $\mu_p$ is a \emph{point measure} supported on a finite or countable subset of $\mathbb{R}$;
    
    \item $\mu_{ac}$ is an \emph{absolutely continuous measure} with density in $L^1_{loc}(\mathbb{R})$ with respect to Lebesgue measure;
    
    \item $\mu_s$ is a \emph{singular continuous measure} (i.e., there exists a Borel set $\Sigma \subset \mathbb{R}$ with Lebesgue measure zero such that $\mu_s(\mathbb{R}\setminus\Sigma) = 0$).
\end{itemize}

Therefore, any system of harmonic oscillators (SHO) admits a decomposition into three distinct components:
\begin{itemize}
    \item A \emph{point} component (finite or countable)
    \item An \emph{absolutely continuous} component
    \item A \emph{singular} component
\end{itemize}

In this work, we investigate the dynamical properties of SHO systems corresponding to each measure type: point, absolutely continuous, and singular. Furthermore, we systematically compare these properties across different measure types, highlighting their distinctive features.

The SHO system \eqref{eq1} constitutes a Hamiltonian system $(E, \omega, h)$, where
$E$ is a real Hilbert space,
$\omega$ is a shift-invariant symplectic form on $E$,
$h$ is the Hamiltonian function (see Appendix).

The space $E$ is obtained through the realization mapping $\mathcal{R}: \mathcal{H} \to E$ of the complex Hilbert space $\mathcal{H} = L^2(X,\mu,\mathbb{C})$, which admits the decomposition:
\begin{equation*}
\mathcal{H} = \text{Re}(\mathcal{H}) \oplus \text{Im}(\mathcal{H}) = L^2(X,\mu,\mathbb{R}) \oplus L^2(X,\mu,\mathbb{R})
\end{equation*}
(see Appendix). The realization map $\mathcal{R}$ is explicitly given by:
\begin{equation}\label{ove}
\mathcal{R}(\phi) = (q, p), \quad \phi \in \mathcal{H},
\end{equation}
where $q(x) = \text{Re}\,\phi(x)$ and $p(x) = \text{Im}\,\phi(x)$ for all $x \in X$.

The symplectic form $\omega$ is a non-degenerate skew-symmetric bilinear form on $E$. The symplectic operator $\mathbf{J}$ is induced by the complex structure of $\mathcal{H}$ through the relation:
\begin{equation*}
\mathbf{J}(\mathcal{R}(\phi)) = \mathcal{R}(i\phi), \quad \phi \in \mathcal{H}.
\end{equation*}
Consequently, the symplectic form can be expressed as:
\begin{equation*}
\omega(\mathcal{R}(\phi), \mathcal{R}(\psi)) = (\mathbf{J}\mathcal{R}(\phi), \mathcal{R}(\psi))_E = \frac{i}{2}\left[(\phi,\psi)_{\mathcal{H}} - (\psi,\phi)_{\mathcal{H}}\right] \quad \forall \phi,\psi \in \mathcal{H}.
\end{equation*}

The Hamiltonian function $h: E \supset E_1 \to \mathbb{R}$, associated with the frequency function $\lambda$ in the SHO definition, is given by:
\begin{equation}\label{H}
h(p,q) = \frac{1}{2}\int_X \lambda(x)(p(x)^2 + q(x)^2) d\mu(x), \quad (p,q) \in E.
\end{equation}

The space $\mathcal{H} = L^2(X,\mu,\mathbb{C})$ (along with its realization $E$) constitutes the phase space of the SHO. The unitary group $\mathbf{V}$ acting on $\mathcal{H}$ induces a phase flow $\Phi$ on the realization space $E$ through the representation:
\begin{equation}\label{P}
\Phi_t = \mathcal{R} \mathbf{V}_t \mathcal{R}^{-1}, \quad t \in \mathbb{R}.
\end{equation}

The phase flow $\Phi$ forms a one-parameter group of orthogonal transformations on $E$ such that each two-dimensional symplectic subspace of $E$ remains invariant under the flow.
In coordinates relative to the symplectic basis, the flow $\Phi$ acts by the rule
$\Phi_t(\mathcal{R}(\phi)) = \mathcal{R}(\mathbf{V}_t\phi).$

Given an SHO \eqref{eq1} and an initial state $\phi \in \mathcal{H}$, the system generates a trajectory $\Gamma_{\phi}$ in phase space defined by:
\begin{equation}\label{3}
\Gamma_{\phi} = \{e^{it\lambda(x)}\phi(x) \mid t \in \mathbb{R}\} \subset \mathcal{H}.
\end{equation}

The study of conservation laws and integrability for linear systems in Hilbert spaces with additional positive integrals has been investigated in \cite{TSh, K22}, particularly focusing on their implications for SHO dynamics.

For the SHO system, we obtain the following structure of conserved quantities. The Hamiltonian system \eqref{H} possesses a complete set of first integrals $I_x: E \to \mathbb{R}$ defined by
\begin{equation*}
I_x(q,p) = \frac{1}{2}\lambda(x)(p(x)^2 + q(x)^2), \quad x \in X.
\end{equation*}

These first integrals determine invariant manifolds of the Hamiltonian flow \eqref{P}, which can be parameterized by $\mu$-measurable positive functions $r: X \to (0,+\infty)$ through the level sets
\begin{equation*}
\mathbb{T}_r = \{(q,p) \in E \mid I_x(q,p) = r(x) \text{ for all } x \in X\}.
\end{equation*}

Each above manifold $\mathbb{T}_r$ represents an invariant torus of the Hamiltonian system \eqref{H}.

For every nonnegative $\mu$-measurable function $r \in L^2(X,\mu)$, the invariant torus $\mathbb{T}_r$ of the SHO system \eqref{eq1} is the following set:
\begin{equation}\label{tor}
\mathbb{T}_r = \{ u \in L^2(X,\mu,\mathbb{C}) \mid \lvert u(x)\rvert = r(x) \text{ for } \mu\text{-almost every } x \in X \}.
\end{equation}

Systems of harmonic oscillators and their associated invariant tori in Hilbert spaces naturally emerge in the study of quantum systems. A quantum system is formally defined as a pair $(\mathcal{H}, \mathbf{U})$, where:
\begin{itemize}
    \item $\mathcal{H}$ is a separable complex Hilbert space,
    \item $\mathbf{U} = \{ \mathbf{U}(t) \}_{t \in \mathbb{R}}$ is a strongly continuous one-parameter unitary group on $\mathcal{H}$.
\end{itemize}

The unitary group $\mathbf{U}$ admits a unique generator $\mathbf{L}$, which is a self-adjoint operator on $\mathcal{H}$, yielding the spectral representation:
$
\mathbf{U}(t) = e^{-it\mathbf{L}}, \quad t \in \mathbb{R}.
$
The system's dynamics are governed by the Schr$\ddot{\rm o}$dinger equation:
$
i \frac{d}{dt} \psi(t) = \mathbf{L} \psi(t).
$

For finite-dimensional Hilbert spaces, the equivalence between the Schr$\ddot{\rm o}$dinger equation and the dynamics of oscillator systems is well-established \cite{MGI}. In the infinite-dimensional case, we obtain the following result through application of the spectral theorem:

\begin{theorem}\cite{Volovich19}.
For any quantum system $({\mathcal H},{\bf U})$, there exists:
\begin{itemize}
    \item A unitary operator ${\bf W}: {\mathcal H}\to L^2(X,\mu)$
    \item A measurable function $\lambda: X\to \mathbb R$
\end{itemize}
such that:
\begin{enumerate}
    \item The Schr$\ddot{\rm o}$dinger equation $i\frac{d}{dt}\psi = {\bf L}\psi$ is unitarily equivalent to the harmonic oscillator system:
    \begin{equation}
    i\frac{\partial}{\partial t}\phi(t,x) = \lambda(x)\phi(t,x), \quad (x,t) \in X\times\mathbb{R}
    \end{equation}
    
    \item The unitary equivalence is implemented by the intertwining relation:
    \begin{equation}
    {\bf W}{\bf U}(t) = {\bf V}_t{\bf W}, \quad t\in\mathbb{R}
    \end{equation}
\end{enumerate}
Consequently, every quantum system is completely integrable, with an infinite family of integrals of motion explicitly determined by this construction.
\end{theorem}

\section{A system of oscillators with a 
measure of a general type}
This section investigates three fundamental characteristics of linear flows on invariant tori such as Periodicity of trajectories, Ergodicity (topological transitivity) and Wandering/non-wandering behavior of points.
We establish conditions for these properties in terms of the characteristic function of the SHO measure and intrinsic properties of the linear flow.

Let $\mathrm{ca}(\mathbb{R})$ denote the Banach space of complex-valued countably additive Borel measures with bounded variation on $(\mathbb{R}, \mathcal{B})$, where $\mathcal{B}$ is the $\sigma$-algebra of  Borel subsets of $\mathbb{R}$. The norm on $\mathrm{ca}(\mathbb{R})$ is given by the total variation of a measure. For any $\mu \in \mathrm{ca}(\mathbb{R})$, we denote by $\hat{\mu}(y)$, $y \in \mathbb{R}$, its Fourier transform.

\begin{definition}
A point $u \in H$ is called \emph{wandering} for the flow $\mathbf{\Phi}$ if there exist:
\begin{itemize}
    \item A neighborhood $O(u) \subset H$ of $u$
    \item A time $T > 0$
\end{itemize}
such that for all $t > T$:
\begin{equation*}
\Phi_t(O(u)) \cap O(u) = \emptyset.
\end{equation*}
\end{definition}

\begin{lemma}\label{le1}
Let $\mu$ be a non-negative measure in $\mathrm{ca}(\mathbb{R})$. If 
\begin{equation}\label{cond}
{\overline {\lim\limits_{\xi \to \infty }}} \lvert \hat{\mu}(\xi)\rvert < \|\mu\|,
\end{equation}
then:
\begin{enumerate}
    \item There exists a non-negative function $u \in H = L^2(\mathbb{R}, \mathcal{B}, \mu, \mathbb{C})$ such that the linear flow $\mathbf{\Phi}$ has no dense trajectory $\Gamma_u = \{\mathbf{\Phi}_t u \mid t \in \mathbb{R}\}$ in the torus $\mathbb{T}_u$;
    
    \item Every point $v \in \mathcal{H}$ of the invariant torus $\mathbb{T}_{\vert u\vert }$ is wandering for the flow 
    \begin{equation}
    (\mathbf{\Phi}_t v)(x) = e^{itx}v(x), \quad x \in \mathbb{R}, \ t \in \mathbb{R}, \ v \in H.
    \end{equation}
\end{enumerate}
\end{lemma}

\begin{proof}
Consider the trajectory defined by the flow:
\begin{equation}
(\Phi_t u)(x) = e^{itx}u(x), \quad t \in \mathbb{R}, \ x \in \mathbb{R},
\end{equation}
with initial condition $u \in H$.

Let $u \in H$ be a function satisfying $\vert u(x)\vert  = 1$ for $\mu$-almost every $x \in \mathbb{R}$. Such functions exist in $H$ since $\mu$ has finite variation. For instance, the constant function $1(x) \equiv 1$ belongs to $H = L^2(\mathbb{R}, \mathcal{B}, \mu, \mathbb{C})$. 
Consequently, we may define the invariant torus:
\begin{equation*}
\mathbb{T}_u = \{ u \in L^2(\mathbb{R}, \mathcal{B}, \mu, \mathbb{C}) :\  \vert u(x)\vert  = 1 \ \mu\text{-a.e.}\}
\end{equation*}
where $\mu\text{-a.e.}$ denotes $\mu$-almost everywhere.

First, we prove that the assumptions of Lemma~\ref{le1} on the Fourier transform $\hat{\mu}$ imply the following dynamical property:
\begin{equation}\label{vne}
\exists \delta = \delta(\mu) > 0,\ \exists T > 0 : \Phi_t u \notin O_\delta(u) \quad \forall t \in (-\infty,-T) \cup (T,+\infty),
\end{equation}
where $O_\delta(u)$ denotes the $\delta$-neighborhood of $u$ in $H$.

We have
$$\|{\bf {\Phi}}(t)u-u\|_H^2=\int\limits_{\mathbb R}\lvert (e^{itx}u(x)-u(x)\rvert ^2d\mu (x)=$$ $$=\int\limits_{\mathbb R}[2-e^{itx}-e^{-itx}]\lvert u(x)\rvert ^2d\mu (x)=2\| \mu \|-\hat \mu (t)-\hat \mu (-t).$$ 
Since ${\overline {\lim\limits_{\xi \to \infty }}}\lvert \hat \mu (\xi )\rvert <\| \mu \|$, there are numbers $\delta _{\mu }={\sqrt {
\| \mu \|-\lvert {\overline {\lim\limits_{t\to \infty } }} \mu (t)\rvert }}>0$ and $T>0$ such that $2\| \mu \|-\hat \mu (t)-\hat \mu (-t)>\delta _{\mu}^2$ for any $t:\, \vert t\vert >T$. Thus, the statement (\ref{vne}) is proved.

We now establish that $u$ is a wandering point for the flow $\mathbf{\Phi}$. Since each mapping $\mathbf{\Phi}_t$ is unitary for all $t \in \mathbb{R}$, it preserves neighborhoods:
$
\mathbf{\Phi}_t(O_a(u)) = O_a(\mathbf{\Phi}_t u) \quad \text{for any } a > 0 \text{ and } t \in \mathbb{R}.
$
From condition~\eqref{vne}, we conclude that for any $a \in (0, \delta/2)$ and $\vert t\vert  > T$:
$
\mathbf{\Phi}_t(O_a(u)) \cap O_a(u) = \emptyset.
$
Hence,  $u$ is wandering point of the  flow ${\bf \Phi}$.

Finally, 
we complete the proof by showing that the trajectory $\Gamma_u = \{\mathbf{\Phi}_t u \mid t \in \mathbb{R}\}$ is not dense in the torus $\mathbb{T}_u$. 

Assume, by contradiction, that $\Gamma_u$ is dense in $\mathbb{T}_u$. Then, in particular, $\Gamma_u$ would be dense in the intersection $O_\delta(u) \cap \mathbb{T}_u$, where $\delta > 0$ is the constant from condition~\eqref{vne}.

However, this contradicts the wandering property established in~\eqref{vne}, which guarantees the existence of $T > 0$ such that:
\begin{equation*}
\Gamma_u \cap O_\delta(u) = \emptyset \quad \text{for all } |t| > T.
\end{equation*}

By condition~\eqref{vne}, the truncated trajectory $\Gamma_u^T = \{\mathbf{\Phi}_t u \mid t \in [-T,T]\}$ would need to be dense in $O_\delta(u) \cap \mathbb{T}_u$. 

Observe that the set $\Gamma_u^T$ is compact in $H$ as the continuous image of the compact interval $[-T,T]$ under the flow map $t \mapsto \mathbf{\Phi}_t u$. Note that $O_\delta(u) \cap \mathbb{T}_u$ is relatively open in $\mathbb{T}_u$.

This leads to a contradiction: a compact set $\Gamma_u^T$ cannot be dense in the open set $O_\delta(u) \cap \mathbb{T}_u$ unless $\mathbb{T}_u$ itself is compact, which is not the case for infinite-dimensional tori. The density assumption must therefore be false.
\end{proof}

It is obvious that  following two statements are
true.

\begin{lemma}\label{le2}
Let $\mu$ be a countably additive non-negative Borel measure on $\mathbb{R}$. For any $u \in L^2(\mathbb{R}, \mu, \mathbb{C})$, the set function $\mu_u$ defined by
\begin{equation*}
\mu_u(A) = \int_A \vert u(x)\vert ^2 d\mu(x) \quad \text{for all Borel sets } A \subset \mathbb{R}
\end{equation*}
is a countably additive non-negative Borel measure.
\end{lemma}

\begin{lemma}\label{le3}
Let $\mu$ be a countably additive non-negative Borel measure on $\mathbb{R}$ and $u \in L^2(\mathbb{R}, \mu, \mathbb{C})$. Then for $v \in L^2(\mathbb{R}, \mu, \mathbb{C})$, the following are equivalent:
\begin{enumerate}
    \item $v \in \mathbb{T}_u$ (i.e., $\vert v(x)\vert = \lvert u(x)\rvert $ for $\mu$-a.e. $x\in\mathbb{R}$)
    \item There exists $\phi \in L^\infty(\mathbb{R}, \mu_u, \mathbb{C})$ with $\lvert \phi(x)\rvert = 1$ $\mu_u$-a.e. such that $v(x) = \phi(x)u(x)$ $\mu$-a.e.
\end{enumerate}
where $\mu_u$ is the measure $d\mu_u = \vert u\vert ^2 d\mu$ from Lemma~\ref{le2}.
\end{lemma}

\begin{theorem}\label{te1}
Let $\mu$ be a countably additive Borel non-negative measure of bounded variation on $\mathbb{R}$ and $u \in L^2(\mathbb{R}, \mu, \mathbb{C})$. If 
\begin{equation}\label{limcond}
{\overline {\lim\limits_{\xi \to \infty }}}\lvert \hat \mu _u (\xi )\rvert  < \|\mu_u\|,
\end{equation}
where $\mu_u = \vert u\vert^2\mu$, then:
\begin{enumerate}
    \item The restricted dynamical system $\mathbf{\Phi}\vert _{\mathbb{T}_u}$ is not topologically transitive;
    
    \item Every point $v \in \mathbb{T}_u$ is wandering for $\mathbf{\Phi}\vert _{\mathbb{T}_u}$.
\end{enumerate}
\end{theorem}

\begin{proof}[Proof of Theorem~\ref{te1}]
The argument combines the results of Lemmas~\ref{le1}, \ref{le2}, and \ref{le3}:

\begin{enumerate}
    \item By Lemma~\ref{le2}, the set function $\mu_u = \vert u\vert ^2\mu$ defines a countably additive Borel measure. The condition \eqref{limcond} shows that $\mu_u$ satisfies the hypotheses of Lemma~\ref{le1}.

    \item For any $v \in \mathbb{T}_u$, Lemma~\ref{le3} guarantees the existence of a measurable function $\phi : \mathbb{R} \to \mathbb{C}$ with $\vert \phi(x)\vert = 1$ $\mu_u$-a.e. such that $v(x) = \phi(x)\vert u(x)\vert $ $\mu$-a.e.

    \item Since $\mu_u$ has bounded variation and $\vert \phi\vert  \equiv 1$, we have $\phi \in L^2(\mathbb{R}, \mathcal{B}, \mu_u, \mathbb{C})$.

    \item Lemma~\ref{le1} then applies to show that the restricted flow $\mathbf{\Phi}\vert _{\mathbb{T}_u}$ is not transitive and  
every $v \in \mathbb{T}_u$ is wandering.
\end{enumerate}
\end{proof}

\begin{theorem}\label{te2}
Let $\mu$ be a countably additive Borel non-negative measure of bounded variation on $\mathbb{R}$. Assume there exists $\sigma \in (0,1)$ such that for every Borel set $A \subset \mathbb{R}$, the Fourier transform of the restricted measure satisfies:
\begin{equation}\label{sigma-cond}
{\overline {\lim\limits_{\xi \to \infty }}}
\vert \widehat{\mu\vert _A}(\xi)\vert  \leq (1-\sigma)\|\mu\vert _A\|.
\end{equation}
Then the linear flow $\mathbf{\Phi}$ exhibits the following non-transitivity properties:
\begin{enumerate}
    \item For any $u \in H = L^2(\mathbb{R}, \mathcal{B}, \mu, \mathbb{C})$ with $\vert u(x)\vert  > 0$ $\mu$-a.e., the trajectory $\Gamma_u = \{\mathbf{\Phi}_t u : t \in \mathbb{R}\}$ is not dense in the torus $\mathbb{T}_u$;
    \item Every point $v \in \mathbb{T}_{\vert u\vert }$ is wandering for the flow defined by:
    \begin{equation*}
    (\mathbf{\Phi}_t v)(x) = e^{itx}v(x), \quad x \in \mathbb{R}, \ t \in \mathbb{R}.
    \end{equation*}
\end{enumerate}
\end{theorem}

\begin{proof}[Proof of Theorem~\ref{te2}]
The result for constant amplitude functions follows from Lemma~\ref{le1}. For $u(x) = c$ (constant), we derive explicit parameters:
\begin{equation*}\label{delta-explicit}
\delta_{\mu,c} = c\sqrt{\|\mu\| - \varlimsup_{t\to\infty} \vert \hat{\mu}(t)\vert} > 0
\end{equation*}
such that for some $T > 0$, the flow satisfies:
\begin{equation}\label{vnec}
\mathbf{\Phi}_t u \notin O_{\delta_{\mu,c}}(u) \quad \forall t \in (-\infty,-T) \cup (T,+\infty).
\end{equation}

This establishes both the non-density of trajectories and the wandering property in this special case.

Consider $u$ to be a simple $\mu$-measurable function of the form:
\begin{equation}
u = \sum_{j=1}^N c_j \chi_{A_j},
\end{equation}
where $A_1,\ldots,A_N \in \mathcal{B}$ are Borel sets and $c_1,\ldots,c_N \in \mathbb{C}$.

Since the flow $\mathbf{\Phi}$ preserves the support of $u$, we can analyze the evolution component-wise:
\begin{equation}\label{flow-decomp}
\|\mathbf{\Phi}_t u - u\|_H^2 = \sum_{j=1}^N \|(\mathbf{\Phi}_t u - u)\vert _{A_j}\|_H^2 = \sum_{j=1}^N \|\mathbf{\Phi}_t u\vert _{A_j} - u\vert _{A_j}\|_H^2.
\end{equation}
This decomposition allows us to study the dynamics on each invariant component $A_j$ separately.

Consequently, there exists a time threshold $T = T(u) > 0$ such that for all $\vert t\vert > T$, the following norm estimate holds:
\begin{equation}\label{norm-est}
\|\mathbf{\Phi}_t u - u\|_H^2 \geq \sigma\sum_{j=1}^N \vert c_j\vert ^2 \mu(A_j) = \sigma\|u\|_H^2,
\end{equation}
where $\sigma \in (0,1)$ is the spectral  parameter from the condition (\ref{sigma-cond}).

Let $v\in H$. Then there is a simple function $u\in H$ such that 
\begin{equation}\label{approx}
\|u - v\|_H \leq \tfrac{1}{4}\sqrt{\sigma}\|v\|_H.
\end{equation}
This implies the norm bound:
\begin{equation}\label{appr2}
\|u\|_H \geq \|v\|_H - \|v - u\|_H \geq (1 - \tfrac{1}{4}\sqrt{\sigma})\|v\|_H.
\end{equation}
Applying the estimate \eqref{norm-est} to $u$, there exists $T = T(u) > 0$ such that for all $\vert t\vert > T$ we have according to (\ref{approx}) and (\ref{appr2}):
\begin{align*}
\|\mathbf{\Phi}_t v - v\|_H^2 &\geq \left(\|\mathbf{\Phi}_t u - u\|_H - \|\mathbf{\Phi}_t(v - u)\|_H - \|v - u\|_H\right)^2 \\
&\geq \left(\sqrt{\sigma}\|u\|_H - 2\cdot\tfrac{1}{4}\sqrt{\sigma}\|v\|_H\right)^2 \geq \\ &\left(\sqrt{\sigma}(1 - \tfrac{1}{4}\sqrt{\sigma}) - \tfrac{1}{2}\sqrt{\sigma}\right)^2 \|v\|_H^2 = \tfrac{\sigma}{16}\|v\|_H^2.
\end{align*}
Therefore, any point $v\in H,\, v\neq 0,$ is wandering point of the flow $\Phi $.
\end{proof}

\begin{corollary}\label{co1}
If there exists a transitive trajectory $\Gamma_u = \{\mathbf{\Phi}_t u \mid t \in \mathbb{R}\}$ of the SHO, then there must exist a Borel set $A \subset \mathbb{R}$ with $\mu\vert _A \neq 0$ such that:
\begin{equation*}
{\overline {\lim\limits_{t\to \infty }}}\lvert {\widehat{ \mu \vert _A}}(t)\rvert =\|\mu \vert _{A}\|>0
\end{equation*}
\end{corollary}

\begin{remark} Let $\mu$ be a point measure, $\mu =\sum\limits_{k}p_k\delta _{\lambda _k}$, where $\{p_k\}\in \ell _1$.
Hence, $\hat \mu (t)=\sum\limits_{k}p_ke^{-i\lambda _kt},\ t\in \mathbb R$. Therefore, ${\overline {\lim\limits_{t\to \infty }}}\lvert {\hat{ \mu }}(t)\rvert \leq \| \{p_k\} \|_{\ell _1}=\|\mu \|$. Thus, the the system may exhibit transitive trajectories in the case ${\overline {\lim\limits_{t\to \infty }}}\lvert {\hat{ \mu }}(t)\rvert = \| \{p_k\} \|_{\ell _1}$.
\end{remark}

\section{Bernoulli measures as an example of singular measures}

We construct examples of singular measures $\mu$ on $\mathbb{R}$ for which the SHO $(\mathbb{R}, \mu, \mathbf{V}_{e^{itx}})$ has no transitive trajectories and no non-wandering points on any non-degenerate invariant torus.
Consider the family of Bernoulli measures $\{\mu_\eta\}_{\eta\in(0,1)}$ on $\mathbb{R}$. Then each $\mu_\eta$ is a  compactly supported Borel probability measure.
There are $\eta  \in (0,1)$ such that $\mu_\eta$ is a purely singular continuous measure.
The Bernoulli measures $\{\mu_\eta\}_{\eta\in(0,1)}$ and their associated $L^2(\mu_\eta)$ spaces are well-studied \cite{Jess, Jorgen, Solom}. These measures exhibit the following characteristics:

\begin{enumerate}
    \item Special Cases:
    \begin{itemize}
        \item $\mu_{1/3}$ is the canonical probability measure on the standard Cantor set
        \item $\mu_{1/2}$ coincides with (a multiple of) Lebesgue measure on $[-1,1]$
    \end{itemize}

    \item Fourier Representation:
    \begin{equation}\label{Four-Bern}
    \hat{\mu}_\eta(t) = \prod_{k=1}^\infty \cos(2\pi t\eta^k), \quad t \in \mathbb{R}
    \end{equation}

    \item Support Properties:
    \begin{itemize}
        \item For $\eta \in (0,1)$, $\mu_\eta$ is a Borel probability measure with the support:
        \begin{equation*}
        {\rm supp}(\mu_\eta) = K_\eta \subset \left[-\frac{\eta}{1-\eta}, \frac{\eta}{1-\eta}\right]
        \end{equation*}
    \end{itemize}
\end{enumerate}

Thus, Bernoulli measures provide explicit examples of singular continuous measures with varying support properties.



For any fixed $\eta \in (0,1)$ and $t > \frac{1}{2\eta}$, define:
\begin{equation*}
k_1(t) = \min\{k \in \mathbb{N} \mid 2\pi t\eta^k \geq \pi\} .
\end{equation*}
This index satisfies the following estimates:
\begin{enumerate}
    \item $\pi \eta \leq 2\pi t\eta^{k_1(t)+1} < \pi ;$
    \item $\pi \eta^2 \leq 2\pi t\eta^{k_1(t)+2} < \pi \eta .$
\end{enumerate}
Consequently, by the strict monotonicity of cosine on $[0,\pi]$:
\begin{equation}\label{mon-cos}
\cos(\pi\eta) < \cos(2\pi t\eta^{k_1(t)+2}) < \cos(\pi\eta^2)\quad \forall \ \eta \in (0,1),\ \forall \ t>{1\over {2\eta }}.
\end{equation}

Thus, we prove, that for any $t>{1\over {2\eta }}$ there is a number $k_*\in \mathbb N$ such that $$\cos ({{\pi }{\eta }})<\cos (2\pi t\eta ^{k_*})<\cos (\pi \eta ^2).$$
The equality (\ref{Four-Bern}) and inequality (\ref{mon-cos}) imply the condition
$$
\forall \ t>{1\over {2\eta }}\ \exists \ k_*\in {\mathbb N}:\ 
\lvert \hat\mu _{\eta }(t)\rvert \leq \lvert \cos (2\pi t\eta ^{k_*})\rvert \leq \min \{ \lvert \cos (\pi \eta ^2)\rvert ,\, \lvert \cos (\pi \eta )\rvert <1,
$$
hence, ${\overline {\lim\limits_{t\to \infty }}}\lvert \hat \mu _{\eta }(t)\rvert <1$.
Therefore, from the Theorem \ref{te1} we obtain the following statement.

\begin{theorem}\label{te5}
For every Bernoulli parameter $\eta \in (0,1)$, there exists a dense set of functions $u \in L^2(\mathbb{R}, \mu_\eta, \mathbb{C})$ with the following properties:
\begin{enumerate}
    \item Every point $z \in \mathbb{T}_u$ is wandering for the linear flow $\mathbf{\Phi}$
    
    \item No trajectory $\Gamma_z = \{\mathbf{\Phi}_t z \mid t \in \mathbb{R}\}$ is dense in $\mathbb{T}_u$
\end{enumerate}
where $\mu_\eta$ is the Bernoulli measure with parameter $\eta$.
\end{theorem}


\section{Dynamics of
a SHO with a point measure
}\label{Dyn-on-inf-dim-tori}
In this chapter, we investigate the dynamics of systems of harmonic oscillators (SHOs) in the special case where the measure $\mu$ in the triplet $(X,\mu,\mathbf{V})$ (see equation \eqref{eq1}) is a point measure.

We focus on SHOs satisfying:
\begin{itemize}
    \item $X = \mathbb{N}$ (the set of natural numbers)
    \item $\mu$ is the counting measure on $\mathbb{N}$
\end{itemize}

In this setting, the Hilbert space becomes:
\begin{equation}
\mathcal{H} = L^2(\mathbb{N},\mu,\mathbb{C}) = \ell^2(\mathbb{C})
\end{equation}

A non-degenerate invariant torus $\mathbb{T}_r$ (as defined in \eqref{tor}) is determined by a sequence of positive numbers $r = \{r_n\}_{n=1}^\infty \in \ell^2$, and takes the form:
\begin{equation}
\mathbb{T}_r = \left\{ \{z_n\}_{n=1}^\infty \in \ell^2 \mid \vert z_n\vert = r_n \text{ for all } n \in \mathbb{N} \right\}
\end{equation}

In this chapter, we examine the dynamics of systems of harmonic oscillators (SHOs) in the discrete case where the measure $\mu$ in the SHO triplet $(X, \mu, \mathbf{V})$ from Definition~\eqref{eq1} is a point measure.

We specifically study SHOs satisfying:
\begin{itemize}
    \item The parameter space $X = \mathbb{N}$ (the set of natural numbers)
    \item The measure $\mu$ is the counting measure on $\mathbb{N}$
\end{itemize}

In this discrete setting, the phase space becomes:
\begin{equation*}
\mathcal{H} = L^2(\mathbb{N}, \mu, \mathbb{C}) \cong \ell^2(\mathbb{C})
\end{equation*}

A non-degenerate invariant torus $\mathbb{T}_r$ (as defined in \eqref{tor}) is parameterized by a positive, square-summable sequence $r = \{r_n\}_{n=1}^\infty \in \ell^2$. This torus is the following set:
\begin{equation}\label{tn}
\mathbb{T}_r = \left\{ z = \{z_n\}_{n=1}^\infty \in \ell^2(\mathbb{C}) \mid \vert z_n\vert  = r_n \text{ for all } n \in \mathbb{N} \right\}
\end{equation}

The realization map $\mathcal{R}$ defined in \eqref{ove} establishes an isomorphism between the complex Hilbert space $\mathcal{H}$ and its real counterpart $E = Q \oplus P$, where:
\begin{itemize}
    \item $Q = P = L^2(\mathbb{N}, \mu, \mathbb{R})$ are real Hilbert spaces
    \item For each $u = \{u_n\}_{n=1}^\infty \in \mathcal{H} = L^2(\mathbb{N}, \mu, \mathbb{C})$, the map yields:
    \begin{equation*}
    \mathcal{R}(u) = (q,p) \quad \text{where} \quad q_n = \operatorname{Re}(u_n), \ p_n = \operatorname{Im}(u_n) \ \forall n \in \mathbb{N}
    \end{equation*}
\end{itemize}

Under this realization, the invariant torus $\mathbb{T}_r$ from \eqref{tn} corresponds to the following subset of $E$:
\begin{equation}\label{tnr}
\mathcal{R}(\mathbb{T}_r) = \left\{ (q,p) \in E \mid q_n^2 + p_n^2 = r_n^2 \ \forall n \in \mathbb{N} \right\}
\end{equation}

A realization ${\mathcal R}$ (\ref{ove}) of the complex Hilbert space $\mathcal H$ mappings a point
$u\in L_2({\mathbb N},\mu, {\mathbb C}))$ to the point  $(q,p)\in E=Q\oplus P$, where  $Q=P=L_2({\mathbb N},\mu ,{\mathbb R})$, and the following system of equalities are satisfied
$u_n=q_n+ip_n \ \forall \ n\in {\mathbb N}$. The realization of the torus  (\ref{tn}) is the  set (\ref{tnr}) in the space $E$.

An infinite-dimensional torus (\ref{tnr}) in the Hilbert space $E=Q\oplus P$ is the countable Cartesian product of circles
\begin{equation}\label{c2}
\mathbb T=C_1\times C_2\times.... \subset E.
\end{equation}
Here $C_k=\{ p_k^2+q_k^2=r_k^2\} $ is the circle in the two-dimensional {plane}
$E_k={\rm span }(f_k,g_k)$. Since $\mathbb T\subset E$,
\begin{equation}\label{32}
\sum\limits_{k=1}^{\infty }r_k^2<+\infty .
\end{equation}

The Hamilton function
\begin{equation}\label{HNI}
H(q,p)=\sum\limits_{k=1}^{\infty }{{\lambda _k}\over 2}(p_k^2+q_k^2), \quad (q,p)\in E_1,
\end{equation}
where $\{ \lambda _k\}$ be a sequence of positive numbers and $E_1=\{ (q,p)\in E:\ \sum\limits_{k=1}^{\infty }\lambda _k(q_k^2+p_k^2)<+\infty \}$, defines
{a} flow on the torus (\ref{tnr}) in the space
$E$ by the following equalities
\begin{equation}\label{Z}
z_k(t)=({\bf V}_tz_{0})_k=e^{i\lambda _kt}z_{k0},\ t\in \mathbb R;\quad z_k=q_k+ip_k.
\end{equation}

This Hamiltonian flow should be considered firstly in the subspace
$E_2=\{ x\in E:\ \sum\limits_{k=1}^{\infty}\lambda _k^2(q_k^2+p_k^2)<\infty \}$ of the space  $E=l_2$. The reason is the following. The tangent vectors of the flow is defined by the gradient of the Hamilton function (\ref{HNI}) and this function is differentiable  on the subspace $E_2$ in its domain $E_1$ (see Appendix).

Let us study {the} existence of periodic trajectories
$\Gamma _{z_0}=\{ z(t),\ t\in \mathbb R \}$, of the Hamiltonian system (\ref{HNI}) on the torus $\mathbb T$.

\begin{theorem}\label{t1} \cite{VS24}
{\it Let the sequence $\{ \lambda _k\}$ is rationally commensurable, i.e. there is a finite
collection of numbers $\lambda _{j_1},...,\lambda _{j_m}$ such that one of them is a linear combination, {with rational coefficients, of the others}.
Let the sequence $\{ {1\over {\lambda _k}} \}$ {be} unbounded. Then for any non-{degenerate} torus $\mathbb T$ every trajectory of Hamiltonian system (\ref{HNI}) is neither periodic, nor dense in the torus $\mathbb T$. }
\end{theorem}

\begin{theorem}\label{t2}\cite{VS24}
{\it Let $\{ \lambda _k\}$ {be} the sequence of frequencies of the Hamiltonian system
(\ref{HNI}). Let $r\in l_2$ be the sequence of positive numbers.
Then every trajectory of the Hamiltonian system (\ref{HNI}) on the torus ${\mathbb T}_r$ is periodic if and only if the sequence  $\{ \lambda _k\}$ satisfies the following condition.
There are the number
$\lambda _0>0$ and the  sequence of natural numbers
$\{ n_k\}$ such that $\lambda _k=\lambda _0n_k,\ k\in \mathbb N$. }
\end{theorem}

\begin{remark}
Theorem \ref{t2}
presents the criterion for the periodicity of a trajectory on an infinite-dimensional (countable-dimensional) torus. This criterion is the analogue of the first part of the Jacobi theorem. The condition of Theorem \ref{t2}
on the sequence of frequencies of a countable SHO is coincides with the condition of strong rational commensurability in the Jacobi theorem in the case of finite set of values of the sequence. In the case of infinite set of values the condition of {Theorem} \ref{t2}
implies the condition of strong rational commensurability.
Theorem \ref{t1} shows that for the periodicity of a trajectory, only one condition of strong rational commensurability is not enough.
\end{remark}

\begin{example}\label{exa}
Let us consider the SHO with point measure and sequence of frequencies   $\lambda _k={1\over {k!}},\ k\in \mathbb N$. The countable family of frequencies $\{ \lambda _k,\ k\in \mathbb N\}$ {satisfies} the strong rational commensurability condition. But it is easy to see (\cite{VS24}) that
there is no periodic trajectory of the SHO.
\end{example}


Thus, for the SHO systems in Example \ref{exa}, the linear flow exhibits the following properties:
\begin{enumerate}
    \item \textbf{Non-Transitivity}: No trajectory is dense in the invariant torus $\mathbb{T}$.

    \item \textbf{Projection Behavior}: For every finite-dimensional symplectic subspace $E' \subset E$:
    \begin{itemize}
        \item The projected trajectory $\pi_{E'}(\Gamma)$ is periodic in $\mathbb{T}' = \pi_{E'}(\mathbb{T})$
        \item These projections cannot be dense in any $\mathbb{T}'$
    \end{itemize}

    \item \textbf{New Trajectory Class}: The flow contains trajectories that are:
    \begin{itemize}
        \item Non-periodic in the full torus $\mathbb{T}$
        \item Non-dense in all finite-dimensional projections $\mathbb{T}'$
        \item Not approximable by finite-dimensional transitive behavior
    \end{itemize}
\end{enumerate}
This establishes a third fundamental type of trajectory behavior, distinct from both periodic trajectories and projectively dense trajectories.

\begin{theorem}\label{A} \cite{VS24}
{\it There are dynamical systems (15) (i.e. there are sequences $\{ \lambda _k\}:\ {\mathbb N}_+\to {\mathbb R}_+$) with the following property of trajectories. Every trajectory of this dynamical system on a non-degenerate invariant torus $\mathbb T$ neither periodic nor has a projection to a symplectic subspace such that this projection of the trajectory is dense in the projection of the torus $\mathbb T$ to this subspace.}
\end{theorem}

\begin{remark}
In the case of finite dimensional space $E$ there are two types of Hamiltonian systems (1). Every trajectory of a system of first type is periodic. For every system of the second type there is a symplectic subspace of the space $E$ such that the projection of any trajectory on this subspace is dense in the projection of the invariant torus which contains the trajectory. If the above symplectic subspace coincides with the space $E$ then the system is called transitive. According to Weyl theorem in this case the flow of the system (1) is ergodic with respect to Lebesgue measure on the invariant torus.
\end{remark}

\medskip

Let us study the existence of a trajectory $\Gamma _{z_0}=\{ z(t),\ t\in \mathbb R \}$, of the Hamiltonian system (\ref{HNI}) such that this trajectory is dense everywhere in a torus
$\mathbb T$.

\begin{theorem}\label{t3}\cite{VS24}
{\it A trajectory of {the} Hamiltonian system
(\ref{HNI}) is dense everywhere on the torus  $\mathbb T$ in the Hilbert space
$E$
if and only if the set of numbers $\{ \lambda _k,
, k\in {\mathbb N}\}$, is not rationally commensurable.}
\end{theorem}

\begin{remark}
Theorem \ref{t3}
is {a} criterion of density everywhere in {an} invariant torus  of
a countable SHO with {Hamilton function} (\ref{HNI}).
This criterion is the {analogue} of the second part of {Jacobi's}
theorem.
\end{remark}

\section{System of oscillators with an absolutely continuous measure}
Consider a system of harmonic oscillator (SHO) described by Eq.~\eqref{eq1}, 
where the configuration space is \( X = \mathbb{R} \). Here, \( \lambda \colon \mathbb{R} \to \mathbb{R} \) 
is a Lebesgue-measurable function. The measure \( \mu \) is either the Lebesgue measure on \( \mathbb{R} \), or a nonnegative, absolutely continuous measure with a density \( \rho \in L^{1}_{\text{loc}}(\mathbb{R}) \).

First, consider the case where $\mu$ is the Lebesgue measure. The corresponding phase space is the Hilbert space $\mathcal{H} = L^2(X, \mu; \mathbb{C}) = L^2(\mathbb{R})$, which describes the Hamiltonian system with Hamiltonian function~\eqref{H}. The system's dynamics are governed by the Hamilton equations:
\begin{equation}\label{heq}
\begin{aligned}
    \dot{q}(t,x) &= \lambda(x) q(t,x), \\
    \dot{p}(t,x) &= -\lambda(x) p(t,x),
\end{aligned}
\quad \text{for } t \in \mathbb{R}, \ x \in \mathbb{R}.
\end{equation}
The space 
$$
\mathcal{H}_1 = \left\{ C(k) \colon \int_{\mathbb{R}} \vert \lambda(k)\vert \, \vert C(k)\vert ^2 \, dk < \infty \right\}
$$
defines the domain of the Hamiltonian function
$$
H(p,q) = \int_{\mathbb{R}} \lambda(k) \big(p(k)^2 + q(k)^2\big) \, dk,
$$
where $C(k) = q(k) + ip(k)$ for $k \in \mathbb{R}$. 
The space
$$
\mathcal{H}_2 = \left\{ C(k) \colon \int_{\mathbb{R}} \vert \lambda(k) C(k)\vert ^2 \, dk < \infty \right\}
$$
serves as the domain of the Hamiltonian vector field $(q,p) \mapsto \mathbf{J} \nabla H(q,p)$.

An example of the above Hamiltonian system is the is presented by sine-Gordon equation which has the form (\ref{heq}) with $\lambda (k)={\sqrt {k^2+m^2}}$ with some number $m\geq 0$.

System of Hamiltonian equations (\ref{heq}) defines the Hamiltonian flow 
\begin{equation}\label{has}
{\bf \Phi }_t(u)=
e^{it\lambda }u,\ t\in {\mathbb R},\ u\in {\mathcal H}.
\end{equation}
The family of invariant manifolds for the Hamiltonian system consists of tori 
$$
\mathbb{T}_r = \left\{ u \in \mathcal{H} \colon \vert u(k) \vert = r(k),\ k \in \mathbb{R} \right\},
$$
parametrized by functions $r \in L^2(\mathbb{R}, [0, +\infty))$. 

Key properties:
\begin{itemize}
    \item Each torus $\mathbb{T}_r$ corresponds uniquely to a function $r \in L^2(\mathbb{R}, [0, +\infty))$
    \item Every function $u \in L^2(\mathbb{R}, \mathbb{C})$ determines a torus $T_{\vert u \vert}$ satisfying $u \in T_{\vert u \vert}$
\end{itemize}

{Second, let us consider tori that correspond to self-adjoint operator with continuous spectrum.}

A representation of an arbitrary unitary group in a complex Hilbert space 
$\mathbb H$ by a system of harmonic oscillators (\ref{eq1}) is realized by using of the spectral theorem. 

According to the spectral theorem, for every self-adjoint operator $\mathbf{L}$ in the space $\mathbb{H}$, there exists a nonnegative measure $\mu$ on $\mathbb{R}$ such that $\mathbf{L}$ is unitarily equivalent to the multiplication operator $\mathbf{X}$ (multiplication by the argument) in the Hilbert space $\mathcal{H} = L^2(\mathbb{R}, \mu, \mathbb{C})$. This implies the spectral representation
$$
e^{it\mathbf{L}} = \mathbf{W}^{-1} e^{itx} \mathbf{W}, \quad t \in \mathbb{R},
$$
where $\mathbf{W} \colon \mathbb{H} \to \mathcal{H}$ is a unitary isomorphism.

Therefore, we may assume that $X = \mathbb{R}$ in Eq.~\eqref{eq1} and that the operator group $\mathbf{V}$ acts on $\mathcal{H}$ via the rule \eqref{2} with $\lambda(x) = x$. The measure $\mu$ can be any nonnegative Borel measure defined on the $\sigma$-algebra $\mathcal{B}$ of Borel subsets.

Let us study conditions of periodicity and transitivity for a trajectory $\Gamma _u$ in the torus $T_u\subset {\mathcal H}=L^2({\mathbb R},{\mathcal B}({\mathbb R}),\mu ,{\mathbb C})$ in the case of arbitrary measure $\mu :\ {\mathcal B}({\mathbb R})\to \mathbb R$.

{\bf Periodicity property}. Now we study the periodicity
property for a trajectory of the Hamiltonian system
(\ref{heq}) with 
$$\lambda (x)=x,\ x\in \mathbb R$$ 
and arbitrary nonnegative measure $\mu$ on a measurable space $({\mathbb R},{\mathcal B}(\mathbb R))$. A trajectory of this system is generated by the one-parametric group of mappings ${\bf U}_t:\ {\mathcal H}\to {\mathcal H},\ t\in \mathbb R$. Here ${\bf U}_t(u)(k)=e^{itk}u(k),\ k\in {\mathbb R}$.

A trajectory with initial point $u \in \mathcal{H}$ is periodic if it satisfies:
\begin{equation}\label{Pe}
    \exists T > 0 \colon \| u(k)(e^{iTk} - 1) \|_{\mathcal{H}} = 0,
\end{equation}
which is equivalent to the condition $\mathbf{U}_T u = u$ for some $T > 0$.

If there exists a finite collection $W = \{\lambda_1, \ldots, \lambda_N\}$ of strongly rationally commensurable frequencies with $\mu(W) > 0$,  then the Hamiltonian system (\ref{heq}) admits a periodic trajectory in the phase space $\mathcal H$ with the initial point $u\in {\mathcal H}$ such that this periodic trajectory belongs to the set $W$.

In the considered representation (where $\lambda (x)=x,\ x\in \mathbb R$) the condition of periodicity for a trajectory $\Gamma _u$ has the form
\begin{equation}\label{uper}
\exists \ T>0:\ \int\limits_{\mathbb R}\vert u(k)(e^{iTk}-1)\vert ^2d\mu (k)=4\int\limits_{\mathbb R}\vert u(k)\vert ^2\sin ^2({{Tk}\over 2})d\mu (k)=0.
\end{equation}

A number $T > 0$ is called a period of the SHO  $(X, \mu, \lambda)$ if the time-$T$ evolution operator $\mathbf{\Phi}_T$ has 1 as an eigenvalue, where $\{\mathbf{\Phi}_t\}_{t\in\mathbb{R}}$ is the Hamiltonian flow in $\mathcal{H}$ defined by \eqref{has}.

{Now we obtain a condition of the absence of a periodic trajectory for a Hamiltonian system (\ref{heq}). 

Let $X=\mathbb R$, $\mu$ be a nonnegative Radon measure on the measurable space $({\mathbb R}, {\mathcal B})$ and $\lambda $ be a measurable function ${\mathbb R}\to {\mathbb R}$.}
In this case the trajectory with an initial point $u\in {\mathcal H}$ is periodic if and only if
$$\ \exists \ T>0:\ 
{\bf \Phi}_Tu-u= 0.
$$

\begin{theorem}\label{Lemma 8}.
 {\it Let the measure $\mu $ in SHO is absolutely continuous with respect to the Lebesgue measure with the density $\rho \in L^1_{{\rm loc}
}({\mathbb R})$. Let $\lambda (k),\ k\in \mathbb R$, be a function which is strictly monotone and continuously differentiable on some interval 
$\Delta \subset R$. Let 
$$
r\in L^2({\mathbb R},\mu ,\ {\mathbb R}_+):\ 
\int\limits_{\Delta}(r(x))^2\rho (x)dx >0.
$$ 
Then every trajectory $\Gamma \in {\mathbb T}_r$ of a Hamiltonian system (\ref{heq}) is not periodic.}
\end{theorem}

\begin{proof}
Assume, by contradiction, that there exists a point $u \in \mathbb{T}_r$ which is periodic for the Hamiltonian system \eqref{heq}. 
Therefore,
according to (\ref{Pe}) there is a number $T>0$ such that 
\begin{equation}\label{prot}
\int\limits_{\mathbb R}\vert u(k)\vert ^2\vert (e^{iT\lambda (k)m}-1)\vert ^2\rho (k)dk=0\ \forall \ m\in \mathbb N.
\end{equation}
On the other hand, $$\int\limits_{\Delta }\vert u(k)\vert ^2\vert (e^{iT\lambda (k)m}-1)\vert ^2\rho (k)dk=2\int\limits_{\Delta }r^2(k)\rho (k)(1-\cos (T\lambda(k)m))dk=$$ $$=2\int\limits_{\Delta}(r(x))^2\rho (x)dx-2\int\limits_{\Delta }\lvert u(k)\rvert ^2\rho (k)\cos (T\lambda(k)m)dk.$$ 

Applying the substitution $k = \lambda^{-1}(\xi)$ for $\xi \in \lambda(\Delta)$, we derive the following expression for the second term:
$$\int\limits_{\Delta }\lvert u(k)\rvert ^2\rho (k)\cos (T\lambda  (k)m)dk=\int\limits_{\lambda (\Delta )}\lvert u(\lambda  ^{-1}(\xi ))\rvert ^2\cos (T\xi m)\rho (\lambda ^{-1}(\xi ))d  (\lambda^{-1}(\xi )).$$ 

We have $\lvert u(\lambda ^{-1}(\xi ))\rvert ^2\rho (\lambda  ^{-1}(\xi )){{d \lambda ^{-1}(\xi )}\over {d \xi}}\in L^1(\mathbb R)$ since $\int\limits_{\mathbb R}\lvert u(\lambda ^{-1}(\xi ))\rvert 2\rho (\lambda ^{-1}(\xi )){{d \lambda ^{-1}(\xi )}\over {d \xi}}d\xi =\int\limits_{\mathbb R}\lvert u(k)\rvert ^2\rho (k)dk$. Then
$\lim\limits_{m\to \infty }\int\limits_{\mathbb R}\lvert u(k)\rvert ^2\cos (T\lambda (k)m)dk=0$ according to Riemann theorem on oscillations. Hence,
 $$\int\limits_{\Delta }\lvert u(k)\rvert ^2\lvert (e^{iT\lambda (k)m}-1)\rvert ^2\rho (k)dk\geq {1\over 2}\int\limits_{\Delta}(r(x))^2\rho (x)dx$$
for all sufficiently large $m\in \mathbb N$.
According to the Theorem assumption we obtain the contradiction with (\ref{prot}). 
\end{proof}

{\bf Transitivity property}
Consider the family of trajectories of the SHO  \eqref{eq1} on an invariant torus, where:
\begin{itemize}
    \item The configuration space is $X = \mathbb{R}$,
    \item The measure $\mu$ is a non-negative Radon measure on the space ${\mathbb R}$. 
    \item The frequency function $\lambda \colon X \to \mathbb{R}$ is $\lambda(k)=k,\ k\in \mathbb R$.
\end{itemize}
The torus $\mathbb{T}_u$ inherits its topology from the embedding into the Hilbert space $\mathcal{H} = L^2(\mathbb{R})$. A trajectory $\Gamma_u \subset \mathbb{T}_u$ is \emph{transitive} if for every $g \in \mathbb{T}_u$ and every $\epsilon > 0$, there exists $t > 0$ such that
\begin{equation}\label{Den}
    \| \mathbf{U}_t u - g \|_{L^2}^2 = \int_{\mathbb{R}} \vert g(k) - u(k) e^{i \lambda(k) t} \vert^2 \, dk < \epsilon.
\end{equation}

\begin{lemma}\label{le9}
Let $\mu$ be absolutely continuous with respect to the Lebesgue measure, with density $\rho \in L^1(\mathbb{R})$. For any nontrivial element $u \in \mathcal{H} = L^2(\mathbb{R}, \mathcal{B}, \mu; \mathbb{C})$, the trajectory $\Gamma_u$ is not dense in the torus $\mathbb{T}_{\vert u\vert }$. 
\end{lemma}

This statement is the consequence of the Theorem \ref{te1} since $\hat \mu (t)=\hat \rho (t)\to 0$ as $t\to \infty$ according to Riemann oscillation's theorem.

\begin{lemma}\label{le10}
Consider an SHO with:
\begin{itemize}
    \item An absolutely continuous measure $\mu$ (w.r.t.\ Lebesgue measure) having density $\rho \in L_{1,\mathrm{loc}}(\mathbb{R})$
    \item A frequency function $\lambda \colon \mathbb{R} \to \mathbb{R}$ that is strictly monotone and $C^1$-smooth on some interval $\Delta \subset \mathbb{R}$
    \item An amplitude profile $r \in L^2(\mathbb{R}, \mu; \mathbb{R}_+)$ with $\int_\Delta r(x)^2 \rho(x) \, dx > 0$
\end{itemize}
Then there is no transitive trajectory on the torus $\mathbb{T}_{r}$.
\end{lemma}

\begin{proof}Under the assumption of Lemma \ref{le10}
the measure $\mu _r$ has the absolutely integrable density $\vert r\vert ^2\rho$ with respect to the Lebesgue measure. Hence, $\mu _r\in {\rm ca}({\mathbb R},{\mathcal B})$.

Let $u,v\in {\mathbb T}_r$. Then, according to Lemma \ref{le3} there exists $\phi \in L^{\infty }({\mathbb R},\mu _r,{\mathbb C})$ with $\vert \phi (x)\vert =1$ $\mu _r$-a.e. such that $u(x)=\phi (x)v(x)$ $\mu _r$-a.e.

Under the assumption of Lemma \ref{le10}
the condition (\ref{Den}) is equivalent to the condition
\begin{equation}\label{Den2}
\exists\ t>0:\ 
\| {\bf U}_tu-v\|_{L_2}=
\int\limits_R\rho (\lambda ^{-1}(p))\lvert 1-e^{ipt}\rvert ^2\lvert r(\lambda ^{-1}(p))\rvert ^2\, \vert {{\partial \lambda ^{-1}(p) }\over {\partial p}}\vert \, dp<\epsilon 
\end{equation}

On the other hand, for every $t>0$ the inequality holds
$$
\int\limits_R\rho (\lambda^{-1}(p))\lvert 1-e^{ipt}\rvert ^2\lvert r(\lambda^{-1}(p))\rvert ^2\ \vert {{\partial \lambda^{-1}(p) }\over {\partial p}}\vert \, dp\geq $$
$$
\int\limits_{\Delta }\lvert 1-e^{ipt}\rvert ^2\rho (\lambda^{-1}(p))\lvert r(\lambda^{-1}(p))\rvert ^2\ \vert {{\partial \lambda^{-1}(p) }\over {\partial p}}\vert \, dp>0.
$$
Under the assumption of Lemma \ref{le10}
the measure $\nu $ with the absolutely integrable density $\rho (\lambda^{-1}(p))\lvert r(\lambda^{-1}(p))\rvert ^2\ \vert {{\partial \lambda^{-1}(p) }\over {\partial p}}\vert $ with respect to the Lebesgue measure satisfies the condition $\mu _r\in {\rm ca}({\Delta},{\mathcal B})$.

Therefore, according to Lemma \ref{le1}, the inversion of the condition (\ref{Den2}) holds. Thus,   there is no transitive trajectory on the torus $\mathbb{T}_{r}$.
\end{proof}

Let us consider the example 
$\lambda(k)={\sqrt {1+k^2}},\ k\in \mathbb R$, of sine-Gordon Hamiltonian system. Let functions $\rho ,\, r$ and the segment $\Delta$ satisfy conditions of Lemma \ref{le10}. 
Then, as it is shown in the proof of Lemma \ref{le10} there is no transitive trajectory on the torus ${\mathbb T}_r$.

\section
{Wandering and non-wandering property of the flow of countable and continuous SHO.}

Let us remember a definition of a wandering point of a flow.

\begin{definition}\cite{Efremova}[Nonwandering point]\label{def:nwpoint}
A point $A$ in the phase space of a flow $\{\mathbf{V}_t\}_{t \geq 0}$ is called \emph{nonwandering} if for every neighborhood $\mathcal{O}(A)$ of $A$ and every $T > 0$, there exists $t > T$ such that
$$
\mathcal{O}(A) \cap \mathbf{V}_t(\mathcal{O}(A)) \neq \emptyset.
$$
Otherwise, $A$ is called a \emph{wandering point} of the flow.
\end{definition}

\begin{theorem}\label{th51}
{\it For the SHO with a point measure $\mu$, every point in the phase space $E$ is nonwandering under the flow $\{\mathbf{\Phi}_t\}_{t\in\mathbb{R}}$ of  Hamiltonian system \eqref{HNI}.
}
\end{theorem}

\begin{proof} Let $z\in E$. Then, there is {a} torus ${\mathbb T}_r$ such that
$z\in {\mathbb T}_r$. Let us fix a number $\epsilon >0$. Then, there is a number $m\in \mathbb N$ such that $\sum\limits_{k=m+1}^{\infty }r_k^2<{{\epsilon ^2}\over 8}$. Therefore, for every point of the torus ${\mathbb T}_r$ the distance between this point and its orthogonal projection on the subspace $F_m=\{ (p,q)\in E:\ p_i=0,\, q_i=0,\ \forall \ i=1,...,m\}$ is {not}
greater than  ${{\epsilon }\over 2}$.

According to Poincare's {recurrence} theorem for {a} dynamics in
{a} finite dimensional subspace, the orthogonal projection
{$F_m^{\bot}$ of} $z_m$ of the point $z$ {into} the subspace $F_m^{\bot}$ {is a} nonwandering point of the flow (\ref{Z}) in the invariant torus
${\mathbb T}_r^m={\mathbb T}_r\bigcap F_m^{\bot}$.
Hence, the point $z$ is {a} nonwandering point of the flow (\ref{Z}).
\end{proof}

\begin{theorem}\label{th52}
Let $(X, \mu)$ with $X = \mathbb{R}$ be a measure space where:
\begin{itemize}
    \item $\mu$ is absolutely continuous with respect to Lebesgue measure
    \item The density $\rho = d\mu/dx$ satisfies $\rho \in L^1_{\mathrm{loc}}(\mathbb{R})$
\end{itemize}
Then every point $u \in \mathcal{H}$ on the invariant torus $\mathbb{T}_{\vert u\vert }$ is wandering for the flow $\{\mathbf{\Phi}_t\}$ defined in \eqref{has}. 
\end{theorem}

\begin{proof}
For any point $u \in \mathcal{H}$, its trajectory $\mathbf{\Phi}_t u$ ($t \geq 0$) satisfies the weak convergence property:
\[
\mathrm{w}\text{-}\!\lim_{t \to \infty} \mathbf{\Phi}_t u = 0,
\]
by the Riemann-Lebesgue lemma for oscillations. This implies the existence of $T > 0$ such that for all $t > T$:
\[
\lvert (\mathbf{\Phi}_t u, u) \rvert < \tfrac{1}{16} \|u\|^2.
\]

Consequently, we obtain the lower bound:
\[
\|\mathbf{\Phi}_t u - u\|_{\mathcal{H}}^2 = 2\|u\|^2 - 2\mathrm{Re}(\mathbf{\Phi}_t u, u) > 2\|u\|^2 - \tfrac{1}{8}\|u\|^2 > \tfrac{15}{16}\|u\|^2,
\]
and thus $\|\mathbf{\Phi}_t u - u\|_{\mathcal{H}} \geq \frac{\sqrt{15}}{4}\|u\|$.

Using the isometric property of $\mathbf{\Phi}_t$, we observe:
\[
\mathbf{\Phi}_t(B_{1/4}(u)) = B_{1/4}(\mathbf{\Phi}_t u)),
\]
where $B_{1/4}(u) := \{ y \in \mathcal{H} : \|y - u\| < \tfrac{1}{4}\|u\| \}$. 

For any $v \in B_{1/4}(u))$, the triangle inequality yields:
\[
\|\mathbf{\Phi}_t v - u\| \geq \|\mathbf{\Phi}_t u - u\| - \|\mathbf{\Phi}_t(v - u)\| > \tfrac{\sqrt{15}}{4} - \tfrac{1}{4} > \tfrac{1}{4},
\]
when $\|u\| = 1$. Thus, $\mathbf{\Phi}_t(B_{1/4}(u))) \cap B_{1/4}(u)) = \emptyset$ for all $t > T$, proving the wandering property.
\end{proof}

Thus, the dynamical behavior of harmonic oscillator systems exhibits a fundamental dichotomy:

\begin{itemize}
    \item \textbf{Countable systems}: For any invariant torus $\mathbb{T}_r$ of a countable SHO system, every point is non-wandering (Theorem~\ref{th51}). This reflects the complete recurrence of discrete-spectrum systems.

    \item \textbf{Continuous systems}: For continuous SHO systems with absolutely continuous measures satisfying $\rho \in L^1_{\mathrm{loc}}(\mathbb{R})$, under the conditions of Theorem~\ref{th52}, all points on invariant tori are wandering under the linear flow.
\end{itemize}

\section{Conclusion}
In this work, we present a comparative analysis of dynamical properties for three classes of harmonic oscillator systems:
\begin{itemize}
    \item Finite-dimensional systems
    \item Countably infinite systems
    \item Continuous (uncountable) systems
\end{itemize}

We establish necessary and sufficient conditions for:
\begin{itemize}
    \item Periodicity of trajectories
    \item Topological transitivity of flows on invariant tori
    \item Wandering property for points of invariant tori
\end{itemize}
\begin{enumerate}
    \item We prove that every point in the phase space is nonwandering for the infinite-dimensional linear flow generated by a countable system of harmonic oscillators.

    \item For continuous systems of harmonic oscillators, we establish that all points on non-degenerate invariant tori are wandering under the linear flow, which holds for a broad class of such systems.
\end{enumerate}
Our analysis reveals a fundamental dichotomy in infinite-dimensional linear flows:
\begin{itemize}
    \item \textbf{Novel trajectory class}: We demonstrate the existence of a previously uncharacterized type of trajectory that:
    \begin{itemize}
        \item Is absent in finite-dimensional systems
        \item Exhibits neither periodicity nor transitivity in any invariant 4-dimensional symplectic subspace
    \end{itemize}

    \item \textbf{Typical behavior}:
    \begin{itemize}
        \item For countable oscillator systems: Transitive trajectories are generic
        \item For continuous oscillator systems: The generic trajectory is neither periodic nor transitive
    \end{itemize}
\end{itemize}

\section{Appendix I. Infinite-dimensional Hamiltonian system}

In the present paper the following description of infinite-dimensional Hamiltonian system is considered (\cite{Chernoff, Khrennikov, KS-18}).

\begin{definition}[Infinite-dimensional Hamiltonian system]\label{def:inf-ham-sys}
An infinite-dimensional Hamiltonian system is characterized by the triplet $(E, \omega, h)$, where:
\begin{itemize}
    \item $E$ is an infinite-dimensional real separable Hilbert space,
    \item $\omega \colon E \times E \to \mathbb{R}$ is a shift-invariant symplectic form (i.e., a non-degenerate, skew-symmetric bilinear form),
    \item $h \colon D(h) \to \mathbb{R}$ is a densely defined, Frechet-differentiable Hamiltonian function on a dense linear submanifold $D(h) \subset E$.
\end{itemize}
\end{definition}

A symplectic form $\omega$ on a real separable Hilbert space $E$ is called \emph{native} if there exists an orthonormal basis $\{e_k\}_{k\in\mathbb{N}}$ of $E$ satisfying the canonical symplectic relations:
\[
\omega(e_{2k-1}, e_{2j}) = \delta_{j,k}, \quad \omega(e_{2k}, e_{2j-1}) = -\delta_{j,k}, \quad \omega(e_{k}, e_{l}) = 0 \text{ otherwise},
\]
for all $j,k \in \mathbb{N}$, where $\delta_{j,k}$ is the Kronecker delta. Such a basis is called a \emph{symplectic basis} for $(E, \omega)$.

For any native symplectic form $\omega$ on $E$, there exists an orthogonal decomposition 
\[
E = P \oplus Q
\]
where $P = \overline{\operatorname{span}}\{g_j\}_{j\in\mathbb{N}}$ with $g_j = e_{2j-1}$ and $Q = \overline{\operatorname{span}}\{f_k\}_{k\in\mathbb{N}}$ with $f_k = e_{2k}$.
Here $\{g_j\}$ and $\{f_k\}$ form complete orthonormal bases for $P$ and $Q$ respectively, satisfying:
\begin{equation}\label{simp}
\omega(g_j, f_k) = \delta_{j,k}, \quad \omega(g_j, g_k) = \omega(f_j, f_k) = 0 \quad \forall j,k \in \mathbb{N}.
\end{equation}
{In particular}, $\{ e_i,\ i\in \mathbb N \}=\{ g_j,f_k;\ j,\, k\in {\mathbb N}\}$ is {a} symplectic basis of the symplectic space $(E,\omega )$
(see \cite{Khrennikov, KS-18}). {The} subspaces  $Q$ and $P$ are called configuration space and momentum space respectively. The space $P$ has the role of the space which is conjugated to the space $Q$
(\cite{Kuksin, Khrennikov, KS-18, SSh-20}).

The bounded bilinear form $\omega$ on $E$ induces a bounded linear operator $\mathbf{J} \in \mathcal{L}(E)$ through the relation:
\begin{equation}\label{eq:omega-operator}
\omega(u, v) = \langle u, \mathbf{J}v \rangle, \quad \forall u,v \in E.
\end{equation}
The operator $\mathbf{J}$ has the following properties:
\begin{itemize}
    \item \textbf{Non-degeneracy}: $\ker \mathbf{J} = \{0\}$
    \item \textbf{Skew-symmetry}: $\mathbf{J}^* = -\mathbf{J}$
    \item \textbf{Isometry}: $\|\mathbf{J}\|_{\mathcal{L}(E)} = 1$ (when $\omega$ is native)
\end{itemize}

Given any orthogonal basis $\{g_j, f_k\}_{j,k\in\mathbb{N}}$ of $E$, the operator $\mathbf{J}$ defined by
\begin{equation}\label{eq:J-def}
    \mathbf{J}g_j = -f_j, \quad \mathbf{J}f_k = g_k \quad (j,k \in \mathbb{N})
\end{equation}
is a non-degenerate skew-symmetric bounded linear operator on $E$. The associated bilinear form
\begin{equation}\label{eq:omega-def}
    \omega(u,v) := \langle u, \mathbf{J}v \rangle
\end{equation}
is a native symplectic form on $E$, for which $\{g_j, f_k\}_{j,k\in\mathbb{N}}$ becomes a symplectic basis.

In a Hamiltonian system $(E,{\omega },h),$ the Hamilton function is {a} functional $h:\ E_1\to \mathbb R$ such that $E_1$ is a dense linear manifold of the Hilbert space $E$ and, moreover, the functional $h$ is {Frechet}-differentiable on a
linear manifold $E_2$ which is dense in the space $E$. The article {\cite{Kuksin}}
describes the properties of an infinite-dimensional Hamiltonian system and its invariant manifolds.

{The} Hamilton equation for {the} Hamiltonian system
$(E,{\omega },h)$ is the following equation $z'(t)=J(h'(z(t))),\, t\in \Delta,$ {in}
the unknown function $z:\ \Delta \to E_2$ where $\Delta $ is some interval of real line.  (\cite{Khrennikov, KS-18}).

A densely defined vector field $\mathbf{v} \colon E_2 \to E$ is called \emph{Hamiltonian} if there exists a Frechet-differentiable function $h \colon E_1 \to \mathbb{R}$ on the dense linear submanifold $E_2 \subset E_1$ such that
\[
\mathbf{v}(z) = \mathbf{J} Dh(z) \quad \forall z \in E_2,
\]
where:
\begin{itemize}
    \item $Dh(z)$ denotes the Frechet derivative of $h$ at $z$
    \item $\mathbf{J}$ is the symplectic operator associated with $\omega$ by (\ref{eq:omega-def})
\end{itemize}

A one-parameter group $g^t,\, t\in {\mathbb R},$ of continuously differentiable self-mappings of the space $E_2$ is called {a} smooth Hamiltonian flow in the space $E_2$ along the Hamiltonian vector field  ${\bf v}:\ E_2\to E$, if  
$${{d }\over {d t}}g^t(q,p)={\bf v}(g^t(q,p)),\quad \forall \ (p,q)\in E_2,\ t\in \mathbb R.$$ 
If the Hamiltonian flow in the space  $E_2$ has {a}
unique continuous extension to the space $E$ then this extension of the flow is called generalized
Hamiltonian flow in the space $E$ along the Hamiltonian vector field
$\bf v$ (also this flow is called the generalized Hamiltonian flow generated by the Hamilton function $h$).

A symplectic operator \(\mathbf{J}\) on the Hilbert space \(E\) defines a bijective mapping  --- referred to as \emph{complexification} --- from \(E\) to the complex Hilbert space \(H\). This mapping ensures that multiplication by the imaginary unit \(i\) in \(H\) corresponds to the action of the operator \(\mathbf{J}\) in \(E\) \cite{Khrennikov}. The inverse mapping \(H \to E\) is termed the \emph{realization} of the complex Hilbert space \(H\) onto the real Hilbert space \(E\). Such a realization transforms the Schr$\ddot{\rm o}$dinger equation into a linear Hamilton equation \cite{GS22, Khrennikov, KS-18, SSh-20}.

Conditions on the vector field $\bf v$ which {are} sufficient {for}
the existence of a flow in the space $E$ along the vector field  are considered in the monograph
S.G. Kreyn \cite{SGK}.
In the case of a Hilbert space $E$ and a linear vector field {a} sufficient condition {for the existence} of the flow is the self-adjointness of the linear operator
$z\to Dh(z)$ in the complexification of the space  $(E,{\bf J})$.
The sufficient conditions for the existence of a Hamiltonian flow which is generated in the space  $E_2$ by {a} nonlinear Schrodinger equation
are described in \cite{Bourgain, 
Ponce}.

Conditions on the Hamiltonian vector field $\mathbf{v}$ that are sufficient for the existence of a flow in the space $E$ along $\mathbf{v}$ are discussed in the monograph by \cite{SGK}. 
In the case where $E$ is a Hilbert space and $\mathbf{v}$ is a linear vector field, a sufficient condition for the existence of the flow is the \textit{self-adjointness} of the linear operator 
\[
z \mapsto Dh(z)
\]
in the complexification of the space $(E, \mathbf{J})$, where $\mathbf{J}$ is the symplectic structure on $E$.
For the specific case of a Hamiltonian flow generated in the space $E_2$ by a nonlinear Schr\"{o}dinger equation, sufficient conditions are detailed in \cite{Bourgain, Ponce}. These typically involve:
\begin{itemize}
    \item Regularity assumptions on the initial data,
    \item Structural conditions on the nonlinearity (e.g., gauge invariance or subcritical growth),
    \item Conservation laws (e.g., mass, energy) ensuring global well-posedness.
\end{itemize}

\section*{Acknowledgments}
The authors express their gratitude to I.\,Ya.~Aref'eva, L.\,S.~Efremova, and E.\,I.~Zelenov for their interest in this work and valuable discussions.  
We are especially indebted to L. Accardi for his meticulous reading of the manuscript and his insightful suggestions, which significantly enhanced the clarity and rigor of the presentation.

\vspace{0.5em}
\noindent
\textbf{Funding:} This work was supported by the Russian Science Foundation (Grant No.~24-11-00039, Steklov Mathematical Institute).

\par\medskip\noindent
{\bf 
Author Contributions} All authors contributed equally.

\par\medskip\noindent
{\bf 
Data Availability} Data sharing not applicable to this article as no datasets were generated or analysed during
the current study.

\par\medskip\noindent
{\bf Conflicts of Interest} No potential conflict of interest was reported by the author(s).

\bibliographystyle{amsplain}

\end{document}